\documentclass[10pt,english,reqno]{amsart}
\usepackage[T1]{fontenc}
\usepackage[latin9]{inputenc}
\usepackage{color}
\usepackage{amsthm}
\usepackage{amstext}
\usepackage{amssymb}
\usepackage{amsfonts}
\usepackage{esint}
\usepackage{paralist}

\makeatletter
\numberwithin{equation}{section}
\numberwithin{figure}{section}
\theoremstyle{plain}
\newtheorem{thm}{Theorem}
  \theoremstyle{plain}
  \newtheorem{prop}[thm]{Proposition}
  \theoremstyle{remark}
  \newtheorem{rem}[thm]{Remark}
  \theoremstyle{plain}
  \newtheorem{lem}[thm]{Lemma}
  \theoremstyle{definition}
  
    \theoremstyle{plain}
    \newtheorem{cor}[thm]{Corollary}

\usepackage[matrix,arrow,tips,curve,ps]{xy}\SelectTips{cm}{10}
\usepackage{enumerate}
\usepackage{comment} 

\theoremstyle{definition}

\numberwithin{thm}{section} 
\numberwithin{equation}{section}

\def\Lcal{{\mathcal L}}
\def\Mcal{{\mathcal M}}
\def\Ncal{{\mathcal N}}
\def\Ocal{{\mathcal O}}

\def\Xcal{{\mathcal X}}

\def\Zcal{{\mathcal Z}}

\def\PP{{\mathbb P}}

\def\ZZ{{\mathbb Z}}

\newcommand{\rk}{\operatorname{rk}}

\makeatother

\usepackage{babel}

\begin{document}


\title{On Schoen surfaces}

\author{C. Ciliberto, M. Mendes Lopes, X. Roulleau}
\thanks{
The second author is a member of the Center for Mathematical
Analysis, Geometry and Dynamical Systems (IST/UTL).   This research was partially supported by FCT (Portugal) through program POCTI/FEDER,
Project PTDC/MAT/099275/2008 and grant  SFRH/BPD/72719/2010.}
\begin{abstract}
We give a new construction of the irregular, generalized Lagrangian, surfaces of general type with $p_g=5, \chi=2, K^ 2=8$,
recently discovered by C. Schoen in \cite {Schoen}. Our approach proves that,  if $S$ is a general Schoen surface, its 
canonical map is a finite morphism of degree 2  onto a canonical surface with invariants $p_{g}=5, \chi=6, \, K^{2}=8$, a complete intersection of a quadric
and a quartic hypersurface in $\PP^{4}$, with 40 even nodes.
\end{abstract}
\maketitle

\section*{Introduction}\label{sec:intro}

Let $S$ be a smooth projective irregular surface. Let \begin{equation}\label{eq:form}
\varphi_{S}:\wedge^{2}H^{0}(S,\Omega_{S})\to H^{0}(S,K_{S})\end{equation}
 be the natural map. We call the vectors of the space $\wedge^{2}H^{0}(S,\Omega_{S})$
the \emph{formal $2$-forms} of $S$. The \emph{rank} of a formal
$2$-form $\omega$ is the minimum dimension of a subspace $V\subseteq H^{0}(S,\Omega_{S})$
such that $\omega\in\wedge^{2}V$ ; it is an even integer.

A famous theorem by Castelnuovo and De Franchis says that there is
a non--zero formal $2$--form $\omega$ of rank $2$
in $\ker(\varphi_{S})$  if and only if there exists an \emph {irrational pencil} of genus $b\ge 2$ on $S$, i.e. a surjective morphism
$f:S\to B$, where $B$ is a smooth genus $b$ curve and there exist
$\omega_{1},\omega_{2}\in H^{0}(B,\omega_{B})$ such that $\omega=f^{*}(\omega_{1})\wedge f^{*}(\omega_{2})$.

Existence of higher rank formal
$2$-forms in $\ker(\varphi_{S})$ are  more rare and their geometric interpretation more difficult (see \cite {Barja}).
E.g., the existence of such forms is  relevant in the study of the
fundamental group of $S$ (see  \cite{Amoros1, Amoros}). 

With a completely different viewpoint in mind (i.e.,Tate and
Hodge conjectures), C. Schoen discovered in \cite{Schoen} remarkable
minimal irregular surfaces of general type  with invariants $p_{g}=5,\ \chi=2,\, K^ 2=16$  (from now on called \emph{Schoen surfaces}).
They enjoy the property that $\ker(\varphi_{S})$ is generated by a formal
$2$--form of rank $4$ (hence they are \emph{generalized Lagrangian surfaces} in the sense of \cite {Barja}).     Furthermore   they also enjoy the property that $p_g=2q-3$, i.e.  $p_g$ is the minimum possible with respect to $q$ for surfaces with no irrational pencils  of genus $b\ge 2$ (see \cite{MP,MPP2, MPP3} for   the existence of  surfaces with $p_g=2q-3$). 

Other interesting topological properties of Schoen surfaces are pointed out in \S \ref {sec:Schoen} below. 

Schoen's construction is as follows. He 
finds a reducible surface $V$, the transverse union of two irreducible components,
inside the principally polarized abelian variety $A\times A$, where $A=J(C)$ and
$C$ is a general  smooth  irreducible genus 2 curve. He shows that $V$ smooths
to the required surface. The smoothing relies on two main tools: Bloch's semiregularity
(enjoyed by $V$ in $A\times A$) and an explicit deformation of varieties of type $A\times A$ to simple principally
polarized abelian varieties in which the class of $V$ stays of Hodge type $(2,2)$.

The aim of this paper is to give a different, slightly more geometric, approach
to Schoen's construction. It will give us more informations, namely:

\begin{thm}(Theorem \ref  {thm:canonical}). Let $S$ be a general Schoen surface. Its 
canonical map $\varphi_{K}:S\to\mathbb{P}^{4}$ is a finite morphism of degree 2 
onto a canonical surface with invariants $p_{g}=5, \chi=6, \, K^{2}=8$ and  40 even nodes, which is a complete intersection of a quadric
and a quartic hypersurface in $\PP^{4}$. The
ramification of $\varphi_{K}$ takes place at the nodes.
\end{thm}

When the canonical map of a surface of general type has degree  $n> 1$ onto a surface, that surface
either has $p_g = 0$ or is itself canonically embedded (see \cite [Th. 3.1]{Beauville1}). Schoen surfaces
provide one more example of the latter, rather rare, case (see \cite {CPT}; see also the recent preprint \cite {Beauville2}).

Our construction, described in \S \ref {sec:diff}, starts from the same reducible surface $V$ considered by
Schoen. It turns out that  (a slight modification of) $V$ is the double cover of a surface $Z$, which has
$40$ nodes and otherwise has normal crossing singularities. The surface $Z$  sits in the closure of the moduli space of 
complete intersections of a quadric and a quartic hypersurface in $\PP^{4}$ and an easy count of parameters
shows that it deforms to a surface $Y$ which is still a complete intersection of a quadric and a quartic, 
and has 40 nodes and no other singularity. Then we show that the 40 nodes are even.
The double cover of $Y$ branched at the $40$ nodes are Schoen surfaces, and counting parameters one sees that in this way one gets them all. 

These ideas can be applied to other similar  situations in order to find more examples of  irregular surfaces, but
we do not dwell on this here. 

The paper is organized as follows. In \S \ref {sec:prelimi} we recall a few useful know facts. In \S \ref {sec:Schoen} we 
recall Schoen's  main result and, using  
of \cite {Amoros1, Amoros}, we discuss some properties of the fundamental group of Schoen surfaces. Finally, \S \ref {sec:diff}
is devoted to our alternative construction.

\subsection*{Notation} We use standard notation in algebraic geometry. Specifically, if
$X$ is a surface with locally Gorenstein singularities (so that the \emph{dualizing} or \emph{canonical}  sheaf $\omega_X$ is a line bundle),
we denote by $K_X$ the divisor class of $\vert \omega_X\vert$ and we set $p_g(X):=h^ 0(X,\omega_X),\, \chi(X):=\chi(\Ocal_X)=\chi(\omega_X),\, K_X^2:= \omega_X^2$, $q(X)=h^ 1(X, \Ocal_X)$. We may drop the indication of $X$ when this is clear from the context. We note that if $X$ has only Du Val singularities, then the above invariants for $X$ and for a minimal desingularization of $X$ coincide.

\section{Preliminaries}\label{sec:prelimi}

\subsection{Surfaces with normal crossing singularities}\label{sec:invar}

We recall a few known facts (see \cite {Calabri1, Calabri2}
and references therein). Let $V=V_{1}\cup \dots\cup V_{n}$ be a reducible, projective surface  such that:\\
\begin{inparaenum}[(i)]
\item the irreducible components $V_{1}, \dots, V_{n}$ are smooth;\\
\item the  \emph{double curves} $C_{ij}:=V_i\cap V_j$ are smooth and  irreducible, and  $V_i,V_j$ intersect
transversally along $C_{ij}$, for $1\leqslant i <j\leqslant n$;\\
\item $V$ has a finite number of \emph{triple points} $T_{ijk}:=V_i\cap V_j\cap V_k$ and
$V$ around $T_{ijk}$ is analytically isomorphic to the surface of equation $xyz=0$ in 
$\mathbb A^ 3$ around the origin, for $1\leqslant i<j<k\leqslant n$. We set $T_{ij}:=\sum_{k\neq i,j} T_{ijk}$ for the \emph{triple point divisor} on $C_{ij}$,  for $1\leqslant i<j\leqslant n$;\\
\item $V$ has no other singularity. 
\end{inparaenum}

Given $V$ as above, one forms the graph $G_V$:\\
\begin{inparaenum}[$\rhd$] 
\item with \emph{vertices} $v_1,\ldots, v_n$ corresponding to the components $V_{1}, \dots, V_{n}$;\\
\item  with \emph{edges} $c_{ij}$ corresponding to the double curves $C_{ij}$, with $1\leqslant i<j\leqslant n$;\\
\item with \emph{faces} $t_{ijk}$ corresponding to the triple points  $T_{ijk}$, with $1\leqslant i<j<k\leqslant n$.
\end{inparaenum}

In the above setting the dualizing sheaf $\omega_V$ is invertible and one has
\begin{equation}\label{eq:dual}
\left .\omega_{V} \right|_{V_i}\cong \omega_{V_i}\otimes \Ocal_{V_i}(\sum_{j\not=i}C_{ij}),\,\, \text{for }\,\, 1\leqslant i\leqslant n, 
\end{equation}
hence
\begin{equation}\label{eq:k2}
K_{X}^{2}=\sum_{i=1}^ n(K_{X_{i}}+\sum_{j\not=i}C_{ij})^{2}.
\end{equation}
Moreover 
\begin{equation}\label{eq:chi}
\chi(\Ocal_V)=\sum_{i=1}^ n \chi(\Ocal_{V_i})-\sum _{1\leqslant i<j\leqslant n}\chi(\Ocal_{C_{ij}})+t(V)
\end{equation}
where $t(V)$ is the number of triple points of $V$, i.e. the number of faces of $G_V$. 

Let 
\[\Phi_V:\bigoplus_{i=1}^ nH^{1}(V_{i},\mathcal{O}_{V_{i}})\to\bigoplus_{1\leqslant i<j\leqslant n}H^{1}(C_{ij},\mathcal{O}_{C_{ij}})\]
be the natural map and let $p_{g}(V)=h^{0}(V,\omega_V)$. Then
\begin{equation}\label{eq:pg}
p_{g}(V)=b_2(G_V)+\sum_{i=1}^ np_{g}(V_{i})+\dim  ({\rm coker} (\Phi_V)).\end{equation}

If $V=X_{0}$
is the central fiber of a projective family of surfaces $f:\mathcal{X}\to\mathbb{D}$,
over a disc $\mathbb D$, with smooth total space $\Xcal$ and smooth fibers $X_{t}=f^{-1}(t)$, for $t\in \mathbb D-\{0\}$, then these smooth fibres have invariants $p_{g}(X_{t})=p_{g}(V)$,
$K_{X_{t}}^{2}=K_{V}^ 2$ and $\chi(\mathcal{O}_{X_{t}})=\chi(\mathcal{O}_{V})$.

If $V$ sits in a family $f:\mathcal{X}\to\mathbb{D}$ as above, one  says that $V$ is \emph{smoothable} and that $f:\mathcal{X}\to\mathbb{D}$ is a \emph{smoothing} of $V$.
Then  the \emph{triple point formula} holds
\begin{equation}\label{eq:tpf}
N_{C_{ij}\vert V_i}\otimes N_{C_{ij}\vert V_j}\otimes \Ocal_{C_{ij}}(T_{ij})\cong \Ocal_{C_{ij}},\,\, \text {for}\,\, 1\leqslant i<j\leqslant n.
\end{equation}

We recall the following definition from  \cite {Friedman}: $V$ is said to be \emph{$d$--semistable} if  
\[
\Ocal_C(-V):=\bigotimes _{i=1}^ n \frac {\mathcal I_{V_i\vert V}} {\mathcal I_{V_i\vert V}\mathcal I_{C\vert V}}\cong \Ocal _C
\]
where $C=\cup_{1\leqslant i <j\leqslant n}C_{ij}$ is the singular locus  of $V$ and the tensor product is taken as $\Ocal_C$--modules. 
If $V$ is smoothable, then $V$ is $d$--semistable (see \cite [Proposition (1.12)] {Friedman}), but the converse is in general false.
In any event, $\Ocal_C(-V)$ is a line bundle on $C$ (see \cite [Proposition (1.10)] {Friedman}). One defines $\Ocal_C(V):=\Ocal_C(-V)^ *$, and 
$\Ocal_{C_{ij}}(V):= \left.\Ocal_C(V) \right |_{C_{ij}}$. Note that $\Ocal_C(V)=\mathcal Ext^ 1_{\Ocal_Z}(\Omega_V,\Ocal_Z)$ is the \emph {$T^1_V$ sheaf} of $Z$ (see \cite [Proposition (2.3)]{Friedman}). 

\begin{lem}\label{dstab} In the above setting, one has
\begin{equation}\label{eq:TPFD}
\Ocal_{C_{ij}}(V)=N_{C_{ij}\vert V_i}\otimes N_{C_{ij}\vert V_j}\otimes \Ocal_{C_{ij}}(T_{ij}).
\end{equation}
Hence  \eqref {eq:tpf} is necessary  for $d$--semistability. If the dual graph of the singular locus  $C$  of $V$
is a tree, i.e.
\[
p_a(C)=\sum_{1\leqslant i <j\leqslant n} p_a(C_{ij}),
\]
then \eqref {eq:tpf} is also sufficient for $d$--semistability.
\end{lem}

\begin{proof}  Formula \eqref {eq:TPFD} is an immediate consequence of the definition of $\Ocal_C(V)$. If the dual graph of $C$ is a tree, then $\Ocal_C$ is the unique line bundle on $C$ whose restriction to each component of $C$ is trivial. \end{proof}

\subsection{Double covers}\label{ssec:double} The contents of this section are well known. We recall them here to fix  notation and terminology. 

Let $X$ be a projective scheme (over an algebraically closed field $k$ of characteristic $p\neq 2$, though we work over $\mathbb C$ in this paper). A \emph{double cover} of $X$ is a scheme $Y$ and a finite morphism $f: X\to Y$ of degree 2. The datum of  such a double cover is equivalent to give two line bundles $\Lcal$, $\Mcal$ on $X$ such that $\Mcal^ {\otimes 2}=\Lcal$ plus a section $s\in H^ 0(X,\Lcal)$. Let $(U_i)_{i\in I}$ be a finite  covering of $X$ over which both $\Lcal$ and $\Mcal$ trivialize, let 
$(\xi_{ij})_{i,j\in I}$ be the corresponding cocycle for $\Mcal$, let $z_i$ be the coordinate 
in  the fibre of $\Mcal$ over $U_i$ and let $(s_i)_{i\in I}$ be the local functions defining $s$. 
Then we have 
\[ z_i=\xi_{ij}z_j, \,\, \text{and} \,\, s_i=\xi^ 2_{ij}s_j\,\, \text{for all} \,\, i,j\in I\]
and the locus $Y$
\[ z^ 2_i=s_i, \,\, \text{for all} \,\, i\in I\]
 in the total space of $\Mcal$ is well defined and, via the natural projection to $X$, is a double cover $f: Y\to X$. The zero locus $B$ of $s$ is the \emph{branch locus} of the covering and $R:=f^ {-1}(B)$ is the \emph{ramification locus}. As schematic counter image of $B$, $R$ has a non--reduced scheme structure. Note that $B$ is not necessarily a Cartier divisor on $X$: e.g. if $s$ is the zero section, then $Y$ is a double structure on $X$. Similarly, if $X$ is reducible, $s$ could be zero on some component of $X$. 

We will need the following lemma, which is a basic step in extending double covers in families:

\begin{lem}\label{lem:ext} Let $f:\mathcal{X}\to\mathbb{D}$ be a projective family over a disc. Let $\Lcal$ be a line bundle on $\mathcal{X}$ and set $\Lcal_0=\left .\Lcal\right|_{X_0}$. Assume there is a line bundle $\Mcal_0$ on $X_0$ such that $\Mcal_0^ {\otimes 2}=\Lcal_0$. Then, up to shrinking $\mathbb D$, there is a line bundle $\Mcal$ on $\mathcal{X}$ such that $\Mcal^ {\otimes 2}=\Lcal$
and $\Mcal_0=\left .\Mcal\right|_{X_0}$. \end{lem}

\begin{proof} Let $\mathcal U=(U_i)_{i\in I}$ be a finite covering of $\Xcal$ 
over which $\Lcal$ trivializes and $\Mcal_0$ trivializes on $\mathcal V=(V_{i})_{i\in I}$ with $V_i=U_i\cap X_0$ for all $i\in I$.  Let  $(\xi_{ij})_{i,j\in I}$ be the cocycle for $\Lcal$ on $\mathcal U$ and let $(\eta_{ij})_{i,j\in I}$ be the cocycle for $\Mcal_0$ on $\mathcal V$. Then 
\[
\eta_{ij}= \sqrt {\left . \xi_{ij} \right|_{V_{ij} } },\,\,  \text{for all} \,\, i,j\in I
\] 
which encodes the choice of a suitable determination of the square root. Then we may choose the same determination of the square root defining
\[
\zeta_{ij}=  \sqrt {\xi_{ij} },\,\,  \text{for all} \,\, i,j\in I
\] 
on $U_{ij}$ for all $i,j\in I$, and this gives the cocycle $(\zeta_{ij})_{i,j\in I}$ defining $\Mcal$ on $\Xcal$.\end{proof}

\subsection{Hypernodes} 
An \emph{hypernode} of an $n$ dimensional variety $X$, with $n\ge 2$, is a point $p$ such that the analytic germ of $(X,p)$ is isomorphic to the quotient singularity $(\mathbb C^ n/\sigma,{\bf 0})$, where 
\[
\sigma: {\bf x}\in \mathbb C^ n\to -{\bf x}\in \mathbb C^ n.
\]
If $n=2$ this is called a \emph{node}, and it is an $A_1$--singularity. A minimal resolution $\tilde X$ of $X$ at $p$ is gotten by a single blow--up. The exceptional divisor $E$ is isomorphic to $\mathbb  P^ {n-1}$ and $N_{E\vert \tilde X}\cong
\Ocal_{\mathbb P^ {n-1}}(-2)$.

If $X$ is a projective variety with hypernodes $p_1,\ldots,p_h$,
and no other singularity, we can consider its minimal desingularization $f: \tilde X\to X$. Then $\tilde X$
has the exceptional divisors $N_1,\ldots, N_h$ contracted by $f$ to the hypernodes $p_1,\ldots,p_h$. Set $N:=\sum_{i=1}^ hN_i$. One says that
 $p_1,\ldots,p_h$ are \emph{even}, if $\Ocal_{\tilde X}(N)$ is divisible by 2 in ${\rm Pic}(\tilde X)$.
This happens if and only if  there is a commutative diagram
\[
\xymatrix@=15pt{
\tilde Y \ar[d]_{\tilde \pi}  \ar[rr]^{g} &&Y  \ar[d]^{\pi} \\
\tilde X \ar[rr]_{f} && X  
}
\]
where $Y,\,  \tilde Y$ are smooth varieties, $\pi,\,\tilde \pi$ are finite morphisms of degree 2, and $\tilde \pi$ is branched at 
$N$, whereas $\pi$ is branched at $p_1,\ldots,p_h$. The counter images of $p_1,\ldots,p_h$ are points
$q_1,\ldots,q_h\in Y$ and $g$ is the blow--up of $Y$ at $q_1,\ldots,q_h$.

\section{Schoen surfaces}\label{sec:Schoen}


Let $V_{1}=A$ be an abelian surface with $C\subset A$
a smooth curve of genus $g\ge 2$. 
One has $N_{C\vert A}\cong \omega_{C}.$ Let $V_{2}=C\times C$ and let $\Delta\subset V_{2}$
be the diagonal. Then $\Delta\simeq C$ and  $N_{\Delta\vert V_2}\cong \omega_{C}^{*}$.
Let $V$ be the reducible surface consisting of $V_{1}\cup V_{2}$
glued along $C\subset V_{1}$ and the diagonal $\Delta\subset V_{2}$. 

\begin{prop}
\label{pro:The-invariants-are} The invariants of $V$ are 
\[p_g=1+g^ 2, \, \chi= g(g-1),\, K^ 2=8g(g-1).\]
\end{prop}

\begin{proof} This follows from \eqref {eq:k2}, \eqref {eq:chi} and \eqref {eq:pg}. The details can be left to the reader. Only note that the map $\Phi_V$ is surjective, since $h^ 1(A,\Ocal_A(-C))=0$ because $C$ is ample on $A$.
\end{proof}

Schoen proves in \cite{Schoen} that:
\begin{thm}
\label{thm:The Schoen surface} If $g=2$, then  $V$ is smoothable to surfaces with a 4--dimensional generically smooth moduli space. 
\end{thm}

\begin{rem} \label{rem:NS} In \cite [Proposition 10.1, (ii)]{Schoen}, it is stated that  for the general Schoen surface $S$ one has $\rk(NS(S))=2$.  As one can directly see with an argument as in \cite {GH}, the right statement is instead  that  $\rk(NS(S))=1$ (that was also pointed to us in \cite {SchoenPrivate}).\end{rem}

It is not known if $V$ is smoothable for $g\ge 3$. This is an  intriguing question, especially for $g=3$ (see Remark \ref {rem:ge3} below).


Schoen surfaces verify $K^ 2=8\chi$. 
Surfaces whose universal cover is $\mathbb{H}\times\mathbb{H}$,
where $\mathbb{H}=\{z\in\mathbb{C}/\Im m(z)>0\}$ is the \emph{Siegel upper-half
plane}, also verify $K^ 2=8\chi$ and have infinite fundamental group. 
Teicher and Moishezon constructed in \cite{Moishezon,Moishezon 2} finitely many families of surfaces with $K^ 2=8\chi$ and finite (even trivial) fundamental group. The following proposition shows a remarkable property of Schoen surfaces: 

\begin{prop}
The universal cover of  a Schoen surface $S$ is not $\mathbb{H}\times\mathbb{H}$.
Since $q(S)=4$,  $\pi_1(S)$ is not finite and 
finite étale covers of  Schoen surfaces  give an infinite number of families of surfaces with
$K^ 2=8\chi$ whose universal cover is not $\mathbb{H}\times\mathbb{H}$.\end{prop}
\begin{proof}

If $S$ has universal cover $\mathbb{H}\times\mathbb{H}$, then it is the quotient of  $\mathbb{H}\times\mathbb{H}$ by a discrete cocompact subgroup $\Gamma$ of ${\rm Aut}(\mathbb{H}\times\mathbb{H})$ acting freely.  By \cite {Matsushima}, either $\Gamma$ is \emph{reducible}, and $S$ is \emph{isogenous} to the product of two curves (i.e. it is a quotient of a product of two curves  by a fixed--point free group action), or $\Gamma$ is \emph{irreducible} and  $S$ is regular. 

The latter case cannot happen, because $q(S)=4$. Also the former case cannot happen. Indeed, in \cite[Proposition 6.1]{SchoenExotic} Schoen proved that a surface
dominated by a map from a product of curves is \emph{Albanese standard},
i.e.~the class of its image into its Albanese variety $A$ sits in the subring of $H^{\bullet}(A,\mathbb{Q})$
generated by the divisor classes. By contrast, by \cite[Theorem 1.1, (iii)]{Schoen}
Schoen surfaces are  \emph{Albanese exotic},  i.e.~not Albanese standard.
\end{proof}

Note that, according to \cite {CF}, Schoen surfaces do not possess any \emph{semi special tensor}. 

If $S$ is a Schoen surface, set $G:=\pi_1(S)$. We denote by $\{G_n\}_{n\in \mathbb N}$ the  \emph{lower central series} of $G$, defined as
\[
G_1=G, \quad G_{n+1}= [G_n,G], \quad \text{for}\quad n\ge 1,
\]
where $[\cdot,\cdot]$ denotes the \emph {commutator subgroup}. The group $G_{\rm ab}=G_1/G_2$ is the \emph{abelianization} of $G$, and in the present case $G_{\rm ab}\cong H^1(S,\mathbb Z)\cong \mathbb Z^ 8$. 

By \cite[Corollary 1.44] {Amoros1}, \cite {Amoros}, 
the group $G_2/G_3\otimes \mathbb C$ is isomorphic to the (dual of the) kernel of  the natural map
 \[
\psi_S:\wedge^{2}H^{1}(S,\mathbb{C})\to H^{2}(S,\mathbb{C}).
\]
The Betti numbers of Schoen surfaces $S$ are $b_{1}=8$ and $b_{2}=22$,
moreover $h^{1,1}(S)=b_{2}-2p_{g}=12$. The space $\wedge^{2}H^{1}(S,\mathbb{C})$
is $28$-dimensional. The map $\psi_S$ respects the Hodge decomposition, hence it is the direct sum of the map
$\varphi_S$ in \eqref {eq:form} and of its conjugate, and of the map
\[
\phi_S: H^{1,0}(S)\otimes H^{0,1}(S)\to H^{1,1}(S).
\]
 We see that  $\dim \ker \phi_S \geq 4$. On the other hand, by \cite[Proposition 2.2.5]{causin},   $\dim \ker \phi_S \leq 5$.
 The general Schoen surface $S$ has no irrational pencil (see Remark \ref {rem:NS}), in particular it has no morphism $f: S\to B$ to a curve $B$ of genus $b\ge 2$. Since this is a deformation invariant property (see \cite {Catanese1}), the same holds for any Schoen surface. Hence, by Castelnuovo--De Franchis' Theorem, the map $\varphi_{S}$ cannot have a kernel of dimension bigger than 1, hence it is surjective with a 1--dimensional kernel. 

We have (see \cite [Proposition 9.1]{Schoen}):

\begin{cor} Let $S$ be a Schoen surface. Then $6\leq\dim(\ker(\psi_S))\leq 7$, hence $6\leq\dim G_2/G_3\otimes \mathbb C \leq 7$ and $G$ is not abelian. \end{cor}

\begin{rem}
Schoen surfaces (and their covers) seem to be the only known surfaces such that both  $\phi_S$ and $\psi_S$ have a non--trivial kernel.\end{rem}

\begin{rem}\label{rem:ge3}  Consider again the reducible surface $V$  for $g\ge 3$.  Suppose $V$ is smoothable and that $S$ is a general surface in a smoothing of $V$. Since
\[ \dim(\wedge^{2}H^{0}(S,\Omega_{S}))=\frac{1}{2}(g+2)(g+1)<p_g,\,\, \text{for}\,\, g\ge 4,\]
we cannot conclude directly that $\varphi_S$ has a non--trivial kernel if $g\ge 4$. Similarly, one computes $h^{1,1}=2(g+2)$, hence we cannot conclude that $\phi_S$ has a non--trivial kernel if $g\ge 4$.  
The borderline case $g=3$ is attractive. If $V$ is smoothable to a surface $S$, then $\phi_S$ has a non--trivial kernel of dimension at least 3, hence, as in the Schoen surface case, the fundamental group $\pi_1(S)$ is not abelian and it would be interesting to understand it. Moreover, either $\varphi_S$ is an isomorphism, or $\varphi_S$ would have a non--trivial kernel. In the former case $S$ would contradict a conjecture to the effect that the Fano
surface of lines of a smooth cubic threefold and the symmetric product
of curves are the only surfaces $S$ such that $\varphi_S$
is an isomorphism (see \cite{MPP} and also \cite {De}). In the latter case, $S$ would again be a {generalized Lagrangian surface} in the sense of \cite {Barja}, and these surfaces are quite rare and interesting on their own.\end{rem}

\section{Different construction of Schoen surfaces} \label{sec:diff} Here we propose an approach to the construction of Schoen surfaces different from the original one.  
It provides us with the following additional bit of information:

\begin{thm} \label {thm:canonical} Let $S$ be a general Schoen surface.
The canonical map $\varphi_{K}:S\to\mathbb{P}^{4}$ of $S$  is a finite morphism of degree 2 
onto a canonical surface
with invariants $p_{g}=5, \chi=6, \, K^{2}=8$ and 40 even nodes, a complete intersection of a quadric
and a quartic hypersurface in $\PP^{4}$. The
ramification of $\varphi_{K}$ takes place at the nodes.
\end{thm}

We start by
looking at the dualizing sheaf $\omega_{V}$, which, by \eqref {eq:dual},  is the bundle obtained by gluing
$\Ocal_{A}(C)$ on $A=V_{1}$ and $\omega_{V_{2}}(\Delta)$ on $V_{2}=C\times C$ along 
$C$: this is possible since the two bundles both restrict to $\omega_C$ on $C$.   Then we modify $\omega_{V}$ by \emph{twisting
by $V_2$}, which means considering the line bundle $\Lcal$ on $V$ obtained by gluing
$\Ocal_{A}(2C)$ on $A$ and $\omega_{V_{2}}$ on $V_{2}$, the
two bundles both restricting to $\omega_C^ {\otimes 2}$ on $C$.   

\begin{rem} Suppose $V=X_{0}$
is the central fiber of a projective family of surfaces $f:\mathcal{X}\to\mathbb{D}$, as in \S~\ref {sec:invar}.  Then $\omega_V=\left .\omega_{\Xcal} \right|_{X_0}$. 
Twisting by $V_2$, as we did,  is the same as considering the line bundle  $\Lcal=\left .\omega_{\Xcal}\otimes \Ocal_\Xcal(V_2) \right|_{X_0}$. 
Note that both $\omega_{\Xcal}$ and $\omega_{\Xcal}\otimes \Ocal_\Xcal(V_2)$ restrict to the canonical bundle on the general surface of the family. Hence $\Lcal$, as well as $\omega_V$, is a \emph{limit} of the canonical bundle of $X_t$ for $t\in \mathbb D-\{0\}$. \end{rem}

\begin{lem}\label{lem:acca}
We have $h^{0}(V,\mathcal{L})=p_g(V)=5$ and the map $\phi_\Lcal: V\to \mathbb P^ 4$ is a morphism.

\end{lem}
\begin{proof} One has a cartesian diagram
\[
\xymatrix@=15pt{
H^ 0(V,\Lcal) \ar[d]_{s_2} \ar[rr]_{s_1}  &&H^{0}(A,\mathcal{O}_{A}(2C))  \ar[d]^{r_1} \\
H^ 0(V_2,\omega_{V_2})\cong H^ 0(C,\omega_C)^ {\otimes 2}  \ar[rr]_{r_2} && H^ 0(C,\omega_C^ {\otimes 2})  
}
\]
where $r_1,\, r_2$ are restriction maps.
One has  $h^{0}(A,\mathcal{O}_{A}(2C))=4$, $h^{0}(V_{2},\omega_{V_{2}})=4$, $r_1$ is surjective since $h^ 1(A, \Ocal_A(C))=0$ and 
$r_2$ is surjective by Noether's theorem.
Since $h^{0}(C,\omega_{C}^ {\otimes 2})=3$,  we have $h^{0}(V,\mathcal{L})=5$. Moreover $p_g(V)=5$ follows from \eqref {eq:pg}, because $h^ 1(A,\Ocal_A(-C))=0$ implies $\Phi_V$ is surjective. Finally, the surjectivity of $r_1$ and $r_2$ implies the surjectivity of both $s_1,\, s_2$, and since $\vert 2C\vert$ and $\vert \omega_{V_2}\vert$ are base point free,  also $\vert \Lcal\vert$ is base point free. \end{proof}

We note that $\phi_{\Lcal}:V\to\PP^{4}$ is composed
with an involution $\iota$ of $V$, which restricts to the
involution $\pm$ on $A$ and to $\mathfrak i\times \mathfrak i$ on $V_{2}$,
where $\mathfrak i$ is  the hyperelliptic involution on $C$. Note that the canonical map of $V_{2}=C\times C$
is a $\ZZ_{2}^{2}$--cover of $\mathbb P^ 1\times \mathbb P^ 1$ 
given by the action of $\mathfrak i$ separately
on each coordinate.  The involution $\iota$ has $46$ isolated fixed points on $V$:\\
\begin{inparaenum}[ $\rhd$]
\item  the $16$ points of order two on $A$, $6$ of which lie on
$C$ and coincide with its Weierstrass points; \\
\item  $36$ points on $V_{2}$, the ones having as coordinates the
Weierstrass points on $C$,  $6$ of them lie on
$\Delta\cong C$ and coincide with the  $6$ Weierstrass points
on $C\subset A$;\\
\item in conclusion $40$ isolated fixed points are in the smooth locus of $V$,
the remaining $6$ are on the double curve $C\cong\Delta$. 
\end{inparaenum}

Accordingly, $W=V/\iota$, is the union of  two components:\\
\begin{inparaenum}[$\rhd$]
\item $\Sigma=A/\pm$, the Kummer surface of $A$, with 16 nodes, 6  on $\Gamma:=C/\mathfrak i\cong \mathbb P^ 1$;\\
\item  $T=V_{2}/\mathfrak i\times \mathfrak i$, with  $36$ nodes, $6$
on $\Gamma'=\Delta/\mathfrak i\times \mathfrak i=
C/\mathfrak i\cong \mathbb P^ 1$;\\
\item $\Sigma$ and $T$ are glued along the double curve
$R$, which is $\Gamma$ on $\Sigma$ and $\Gamma'$ on $T$, in such a way that the $6$ nodes located
there coincide in the obvious way and the tangent cones to two coinciding nodes have in common only the tangent line to $R$;\\
\item $W$ has $40$ more nodes off $R$. 
\end{inparaenum}

Next we modify $W$ in order to make it with normal crossing singularities.
To do this, we minimally resolve the singularities of both $\Sigma$ and $T$.
This produces two surfaces $\Sigma',\, T'$.  We abuse notation and still denote by $\Gamma$
and $\Gamma'$ the proper transforms of these curves on $\Sigma',\, T'$. Then we glue $\Sigma',\, T'$ along
$\Gamma$ and $\Gamma'$, and call again $R$ the double curve of the reducible surface
$W'=\Sigma'\cup  T'$ thus obtained. Note that:\\
\begin{inparaenum}[$\rhd$]
\item $\Sigma'$ has  $(-2)$--curves $N_1,\ldots,N_{16}$ and we may assume that  $N_{11},\ldots, N_{16}$ intersect $\Gamma$;\\
\item $T'$ has  $(-2)$--curves $M_1,\ldots,M_{36}$ and we may assume that  $M_{31},\ldots, M_{36}$ intersect $\Gamma'$;\\
\item in conclusion $W'$ has the $(-2)$--curves $N_1,\ldots, N_{10}, M_1,\ldots, M_{30}$,
whereas the curves $N_{10+i},\, M_{30+i}$ meet each other and the double curve $R$ at a point $x_i$, for $1\leqslant i\leqslant 6$. 
\end{inparaenum}

Finally we form a new surface $Z'$ by sticking 
$6$ planes $P_{i}\cong \mathbb P^ 2$ in $W'$ in the following way: $P_i$ contains the two curves
$N_{10+i},\, M_{30+i}$ as lines meeting at  $x_i$, for $1\leqslant i\leqslant 6$. 
The surface $Z'$ has normal crossing singularities and it respects 
the triple point formula \eqref {eq:tpf}. 
We will also consider the surface $Z$ with 40 nodes obtained by $Z'$ by contracting 
the $(-2)$--curves $N_1,\ldots, N_{10}, M_1,\ldots, M_{30}$ to nodes $n_1,\ldots, n_{10}, m_1,\ldots, m_{30}$. 

\begin{lem}
The surfaces $Z, \, Z'$ have invariants
\[
p_{g}=5,\, \chi=6, \, K^ 2=8.\]
\end{lem}
\begin{proof} It suffices to compute the invariants for $Z'$. 
The surface $\Sigma'$  is a $K3$. Moreover
$H^{0}(T,\omega_{T})$ is the space of invariants
of $H^ 0(V_2,\omega_{V_2})\cong H^{0}(C,\omega_{C})^{\otimes2}$ under the hyperelliptic involution
$\mathfrak i$ on $C$. Since $\mathfrak i$ changes the sign of holomorphic
1--forms on $C$, we have
 \[
H^{0}(T,\omega_{T})\cong H^{0}(C,K_{C})^{\otimes2},\]
 hence $T$ (and also $T'$) has $p_{g}=4$. 
The same argument shows
that $T$ has $q=0$. The assertion  $p_g(Z')=5$ follows from \eqref {eq:pg}, by noticing that $b_2(G_{Z'})=0$
and ${\rm coker} (\Phi_{Z'})=0$ because the double curves of $Z'$ are all rational. 

The computations of $K^ 2$ and $\chi$ follow in a similar way by 
\eqref {eq:k2} and \eqref {eq:chi}. \end{proof}

We could consider $\omega_{Z'}$, but as above this is not quite the right
thing to do, because, among other things, this sheaf is negative on the planes
$P_i$, for $1\leqslant i\leqslant 6$.  Rather we consider its twist $\Ncal$ by $T'$, which restricts to:\\
\begin{inparaenum}[$\rhd$]
\item  the line bundle $\Ocal_{\Sigma'}(2\Gamma+\sum_{i=1}^ 6N_{10+i})$ on $\Sigma'$; \\
\item the canonical bundle $\omega_{T'}$ on $T'$;\\
\item the trivial bundle on each of the planes $P_{i}$, for $1\leqslant i\leqslant 6$.
\end{inparaenum}

One has:\\
\begin{inparaenum}[$\rhd$]
\item  the linear system  $\vert 2\Gamma+\sum_{i=1}^ 6N_{10+i}\vert$ on $\Sigma'$ is base point free and birationally maps $\Sigma'$ to the quartic  Kummer surface $\Sigma\subset \mathbb P^ 3$, by contracting $N_1,\ldots, N_{16}$ to the nodes of $\Sigma$;\\
\item the canonical system $\vert \omega_{T'}\vert$ is base point free and we have a commutative diagram
of morphisms
\[
\xymatrix@=15pt{
T' \ar[rrd]_{h_{T'}}  \ar[rr]^f &&T  \ar[d]^{h_T} &&\ar[ll]_g  V_2=C\times C\ar[lld]^{h_{V_2}}\\
   &&Q\subset \mathbb P^ 3  
}
\]
where $Q\cong \mathbb P^ 1\times \mathbb P^ 1$ is a smooth quadric, $f$ is birational, $h_T,\, h_{T'}$ and $g$ have degree 2, and $h_{V_2}$ has degree 4. 
\end{inparaenum}

\begin{lem}\label{lem:zeta1}  We have $h^{0}(Z',\mathcal{N})=p_g(Z')=5$ and the map $\phi_\Ncal: Z'\to \mathbb P^ 4$ is a morphism factoring through a morphism $\phi: Z\to \mathbb P^ 4$, 
whose image $\bar Z$ is the union of a Kummer surface $\Sigma$ lying in a hyperplane $\Pi$ and of a (double) quadric $Q$ lying in another hyperplane $\Pi'$, and $\Sigma$ and $Q$ meet along a conic $\Gamma$ which is a plane section of $Q$ and passes through 6 nodes of $\Sigma$. 
\end{lem}

\begin{proof} We have a cartesian diagram
\[
\xymatrix@=15pt{
H^0(Z,\mathcal N) \ar[d]_{s_2} &  \ar[rr]_{s_1}&&& \quad H^{0}(\Sigma',\Ocal_{\Sigma'}(2\Gamma+\sum_{i=1}^ 6N_{10+i}) )  \ar[d]^{r_1} \\
H^ 0(T',\omega_{T'})& \ar[rr]_{r_2} &&& \quad H^ 0(\Gamma, \Ocal_{\Gamma}\otimes \mathcal N)\cong H^ 0(\mathbb P^ 1, \Ocal_{\mathbb P^ 1}(2))
}
\]
where $r_1,\, r_2$ are restriction maps, both surjective. The proof goes as the one of Lemma \ref {lem:acca}.\end{proof}

\begin{lem}\label{lem:zeta2}  Notation as in  Lemma \ref {lem:zeta1}. Then $\bar Z$ is the complete intersection of the quadric $\Pi\cup \Pi'$ and of a quartic hypersurface. \end{lem}

\begin{proof} We may choose homogeneous coordinates $(x_0:\ldots:x_4)$ in $\mathbb P^ 4$ so that $\Pi$ has equation $x_0=0$ and $\Pi'$  equation $x_1=0$. Suppose that the equation of $\Sigma$ in $\Pi$ is $F(x_1,\ldots,x_4)=0$ and the equation of $Q$ in $\Pi'$ is $G(x_0,x_2,\ldots, x_4)=0$. We may write
\[
G(x_0,x_2,\ldots, x_4)=x_0^ 2+x_0q_1(x_2,x_3,x_4)+q_2(x_2,x_3,x_4)
\]
where $q_1,q_2$ are homogeneous polynomials of degree given by the index.  We may assume that
\[
F(0,x_2,x_3,x_4)=G^2(0,x_2,x_3,x_4)=q^ 2_2(x_2,x_3,x_4).
\]
Consider the homogeneous polynomial of degree $4$
\[
H(x_0,\ldots,x_4)=\sum_{i=0}^ 4 x_0^{4-i} f_i(x_1,\ldots,x_4)
\]
where
\[
f_0=1,\,\, f_1= 2q_1, \,\, f_2=q_1^ 2+2q_2,\,\,  f_3=2q_1q_2,\,\, f_4=F.
\]
The quartic $H=0$ intersects $\Pi$ in $\Sigma$ and $\Pi'$ in the quartic with equation
\[
x_0^ 4+2x_0^ 3q_1+x_0^ 2(q_1^ 2+2q_2)+2x_0q_1q_2+q_2^ 2=0
\]
which is the double quadric $G^ 2=0$. The assertion follows.\end{proof}

Lemma \ref {lem:zeta2} shows that the 40--nodal (and otherwise normal crossings) surface $Z$ sits on the boundary of a partial compactification $\mathfrak M$ of the moduli space of complete intersections of a quadric and a quartic in $\PP^{4}$, which are canonical surfaces with invariants $p_{g}=5,\, \chi=6,\, K^{2}=8$. 
One has  $\dim(\mathfrak M)=10\chi-2K^{2}=44$ moduli. 
In $\mathfrak M$  each node imposes, as well known, one condition at most, and therefore $Z$ is contained in an irreducible, locally closed  subset $\mathcal Z\subset \mathfrak M$ of dimension $\dim(\Zcal)\ge 4$
of $40$--nodal surfaces.

\begin{lem}\label{lem:smooth} The general surface in $\mathcal Z$ has 40 nodes and no other singularity. \end{lem}

\begin{proof} The reducible surfaces $Z$ depend
on $3$ moduli (i.e. the moduli of $C$). So they fill up a proper subvariety $\Zcal'$ of $\Zcal$. The local to global Ext spectral sequence gives the exact sequence
\begin{equation}\label{eq:ltg}
0\to H^ 1(Z,\Theta_Z)\to {\rm Ext}^ 1_{\Ocal_Z}(\Omega^ 1_Z,\Ocal_Z)\to H^ 0(Z, {\mathcal E}xt^ 1_{\Ocal_Z}(\Omega_V,\Ocal_Z))\cong \mathbb C\oplus \mathbb C^ {40}.
\end{equation}
To explain the last isomorphism, note that ${\mathcal E}xt^ 1_{\Ocal_Z}(\Omega_V,\Ocal_Z)$ is supported at the singular locus of $Z$, which consists of the double curve $D:=\Gamma+\sum_{i=1}^ 6(N_{10+i}+M_{10+i})$ plus the 40 nodes $n_1,\ldots, n_{10}, m_1,\ldots, m_{30}$. By Lemma \ref {dstab} (which clearly applies to this case, though $Z$ is singular off the double curve), one has 
\[
{\mathcal E}xt^ 1_{\Ocal_Z}(\Omega_V,\Ocal_Z)\otimes \Ocal_D\cong \Ocal_D.
\]
Moreover 
\[
{\mathcal E}xt^ 1_{\Ocal_Z}(\Omega_V,\Ocal_Z)\otimes \Ocal_{z}\cong \Ocal_z,\,\, {\rm for}\,\ z=n_1,\ldots, n_{10}, m_1,\ldots, m_{30}.
\]
Recall that the vector spaces in \eqref {eq:ltg} have the following meaning:\\
\begin{inparaenum}[$\rhd$]
\item  $H^ 1(Z,\Theta_Z)$ is the tangent space to \emph{locally trivial} deformations of $Z$;\\
\item  ${\rm Ext}^ 1_{\Ocal_Z}(\Omega^ 1_Z,\Ocal_Z)$ is the tangent space of all deformations of $Z$.
\end{inparaenum}
Consider the kernel ${\bf K}$ of the projection 
\[{\rm Ext}^ 1_{\Ocal_Z}(\Omega^ 1_Z,\Ocal_Z)\to H^ 0(Z, \bigoplus_{i=1}^ {10}\Ocal_{n_i}\oplus \bigoplus_{i=1}^ {30}\Ocal_{m_i})\cong \mathbb C^ {40}\]
which is the tangent space to deformations of $Z$ keeping the 40 nodes $n_1,\ldots, n_{10}, m_1,\ldots, m_{30}$, i.e. it is the tangent space to $Z$ in $\Zcal$. The sequence \eqref {eq:ltg} can be replaced by
\[
0\to H^ 1(Z,\Theta_Z)\to\mathbf K \to H^ 0(D, \Ocal_D)\cong \mathbb C.
\]

Let us take now a deformation $f: \Xcal \to \mathbb D$ of $Z$ inside $\Zcal$ parametrized by a disc $\mathbb D$, which is not tangent to $\Zcal'$, in particular it is not a locally trivial deformation of $Z$. Then the tangent vector  to this deformation
is an element in $\mathbf K$ not in $H^ 1(Z,\Theta_Z)$, hence it maps to a non--zero element in $H^ 0(D, \Ocal_D)$. By a (suitable version of) \cite [Proposition (2.5)]{Friedman}, one may assume (up to shrinking $\mathbb D$) that $\Xcal$ is smooth off the curve $A$ described by the deformations of the 40 nodes. The assertion follows.
\end{proof}

Let us consider the desingularization $\mathcal Y\to \Xcal$, which is obtained by blowing--up $\Xcal$ along the singular curve $A$ (see proof of Lemma \ref{lem:smooth}). By composing with $f$ we have a new family $g:\mathcal Y\to \mathbb D$ which is a smoothing of $Z'$. We denote by $E$ the exceptional divisor over $A$. It intersects the general surface $Y_t$ of the family, for $t\neq 0$, in the $(-2)$--curves  deforming $N_1,\ldots, N_{10}, M_1,\ldots, M_{30}$ on $Z'$.

\begin{lem}\label{lem:even} The 40 nodes on the general surface of $\Zcal$ are even.
\end{lem}

\begin{proof} Consider the divisor $E+P$ on $\mathcal Y$, where $P=\sum_{i=1}^ 6 P_i$ (we abuse notation here and denote by $P_i$ its strict transform on $\mathcal Y$, for $1\leqslant i\leqslant 6$).  We note that $\Ocal_{Z'}(E+P)$ is divisible by $2$ in ${\rm Pic}(Z')$. Indeed:\\
\begin{inparaenum}[(i)]
\item  \label {i} $\Ocal_{P_i}(E+P)\cong \Ocal_{\mathbb P^ 2}(-2)$, for $1\leqslant i\leqslant 6$;\\
\item \label {ii} $\Ocal_{\Sigma'}(E+P)\cong \Ocal_{\Sigma'}(N_1+\ldots+N_{16})$, which is divisible by two, because the 16 nodes of the Kummer surface are even;\\
\item \label {iii} $\Ocal_{T'}(E+P)\cong \Ocal_{T'}(M_1+\ldots+M_{36})$, which is also divisible by two, because the 36 nodes of $T'$ are even.\\
\end{inparaenum}
Moreover the halves of the bundles appearing in \eqref {i}, \eqref {ii} and \eqref {iii} above naturally glue to give a line bundle  $\Mcal_0$ on $Z'$ such that $\Mcal_0^ {\otimes 2}=\Ocal_{Z'}(E+P)$. Then, by Lemma \ref {lem:ext}, up to shrinking $\mathbb D$, we may assume that there is a line bundle $\Mcal$ on $\mathcal Y$ such that $\left. \Mcal\right|_{Z'}=\Mcal_0$ and $\Mcal^ {\otimes 2}=\Ocal_{\mathcal Y}(E+P)$. Since $\Ocal_{Y_t}(E+P)=\Ocal_{Y_t}(E)$ for $t\neq 0$, the assertion follows. \end{proof}

We are now in position to finish the:

\begin{proof}[Proof of Theorem \ref {thm:canonical}]
If $Y\in  \Zcal$ is the general surface, we can consider the double cover $\pi: S\to Y$ branched at the 40 nodes of $Y$. The surface $S$ is smooth and one computes its invariants to be
the same as for Schoen surfaces. Moreover $\pi^ *(\omega_Y)=\omega_S$. Next  we have to show that these surfaces are indeed
Schoen surfaces, i.e. they come from smoothings of surfaces of type $V$. 

The  proof of Lemma \ref {lem:even} shows that there is a commutative diagram 
\[
\xymatrix@=15pt{
\mathcal S'\ar[rrd]  \ar[rr]^{\pi'} &&\mathcal Y  \ar[d]^{g} \\
   &&\mathbb  D  }
\]
where $\pi'$ is a double cover branched along $E+P$. Note that $\mathcal S'$ is smooth, because so is $E+P$.
Let $E'+P'$ be the ramification divisor on $\mathcal S'$. Note also that the central fibre of $\mathcal S'$, which is a double cover of $Z'$, is nothing but $V$ plus 6 double planes $P'_i$ whose sum is $P'$, each covering one of the planes $P_i$, for $1\leqslant i\leqslant 6$.

Next we simultaneously contract $E+P$ and $E'+P'$, thus getting a new commutative diagram
\[
\xymatrix@=15pt{
&&\mathcal S' \ar[rrd]^ {\pi'}\ar[lld] \\
\mathcal S\ar[rrd]_h  \ar[rr]^\pi &&\mathcal  \Xcal' \ar[d]^{h} &&\ar[ll]  \mathcal Y \ar[lld]^{g}   \\
 &&\mathbb  D  }
\]
where:\\
\begin{inparaenum}[$\rhd$]
\item $\mathcal S'\to \mathcal S$ is the contraction of $E+P$ and $\mathcal S$ is smooth;\\
\item $\mathcal Y\to \mathcal X'$ is the contraction of $E'+P'$ and $\Xcal' $ has 6 hypernodes arising from the contraction of the six components of $P$ and a curve $A$ of double points coming from  the contraction of $E$;\\
\item $\pi: \mathcal S\to \mathcal Y$ is ramified along $A$  and along the 6 hypernodes;\\
\item the family $h: \mathcal S\to \mathbb D$ is a smoothing of the reducible surface $V$ as dictated by \ref {thm:The Schoen surface}. 
\end{inparaenum}

To finish our proof we have to show that in this way we do get all Schoen surfaces. By Theorem  \ref {thm:The Schoen surface}, Schoen surfaces depend 
on 4 moduli. On the other hand, the double covers we found here depend on $\dim(\mathcal Z)\ge 4$ moduli. This proves our assertion.
\end{proof}

\begin{rem} It is worth stressing that our approach does give an alternative way of proving the existence of Schoen surfaces and of finding their number of moduli. In other words, we do not need to rely on Theorem \ref {thm:The Schoen surface}. Indeed, the argument of the proof of Theorem \ref {thm:canonical}, shows that there are smoothings of $V$, depending on $\dim(\mathcal Z)\ge 4$ moduli.  It takes a few lines in \cite [\S 2] {Schoen} to compute the cohomology of $\Theta_V$ and one has $h^ 1(V, \Theta_V)=3$. Then we have the exact sequence
\[
0\to H^ 1(V,\Theta_V)\to {\rm Ext}^ 1_{\Ocal_V}(\Omega^ 1_V,\Ocal_V)\to H^ 0(V, {\mathcal E}xt^ 1_{\Ocal_V}(\Omega_V,\Ocal_V))\cong H^ 0(C, \Ocal_C)\cong \mathbb C\]
and we prove here that the rightmost map is non--zero. This shows  that $\dim({\rm Ext}^ 1_{\Ocal_V}(\Omega^ 1_V,\Ocal_V))=4$ and that the deformations in ${\rm Ext}^ 1_{\Ocal_V}(\Omega^ 1_V,\Ocal_V)$ are unobstructed.  In addition we have $\dim({\rm Ext}^ 1_{\Ocal_V}(\Omega^ 1_V,\Ocal_V))\ge \dim(\mathcal Z)\ge 4$,  which proves that  
$\dim(\mathcal Z)= 4$. \end{rem}

\bigskip{}

\begin{minipage}{13.0cm}

\parbox[t]{5.5cm}{Ciro Ciliberto\\
Dipartimento di Matematica,\\ II Universit\`a di Roma,\\
Via della Ricerca Scientifica, 00133, Roma, Italia\\
cilibert@axp.mat.uniroma2.it}

\vskip1.0truecm

\parbox[t]{6.5cm}{Margarida Mendes Lopes\\
Departamento de  Matem\'atica\\
Instituto Superior T\'ecnico\\
Universidade T{\'e}cnica de Lisboa\\
Av.~Rovisco Pais\\
1049-001 Lisboa, PORTUGAL\\
mmlopes@math.ist.utl.pt
 } \hfill
 \vskip1.0truecm
 
\parbox[t]{6.5cm}{Xavier Roulleau\\
Universit\'e de Poitiers,\\
Laboratoire de Math\'ematiques et Applications, UMR 7348 du CNRS,\\
 Boulevard Pierre et Marie Curie,\\
T\'el\'eport 2 - BP 30179,\\
86962 Futuroscope Chasseneuil,\\
France\\
xavier.roulleau@math.univ-poitiers.fr}

\end{minipage}

\end{document}